\documentclass[final,3p,times]{elsarticle}
\usepackage{color,ulem}
\usepackage{graphicx,url}

\usepackage{amssymb}
\usepackage{amsmath}
\usepackage{amsthm}

\theoremstyle{plain}
\newtheorem{theorem}{Theorem}[section]
\newtheorem{corollary}[theorem]{Corollary}
\newtheorem{example}[theorem]{Example}

\newtheorem{lemma}[theorem]{Lemma}
\newtheorem{definition}[theorem]{Definition}

\usepackage[english,francais]{babel}

\definecolor{darkblue}{rgb}{0,0,0.7}

\newcommand{\PC}{\color{darkblue}}
\newcommand{\CP}{\color{black}}
\makeatletter
\def\@oddfoot{\footnotesize\itshape
       Preprint to appear in Comptes Rendus M\'ecanique\hfill 
       \textrm\thepage\hfill\hfill Accepted: 29 juillet 2010}%
\def\@evenfoot{\@oddfoot}

\def\ps@pprintTitle{%
     \let\@oddhead\@empty
     \let\@evenhead\@empty
     \def\@oddfoot{\footnotesize\itshape
       Preprint to appear in Comptes Rendus M\'ecanique \hfill Published by Elsevier Masson SAS on behalf of Acad\'emie des sciences}%
     \let\@evenfoot\@oddfoot}
\def\@seccntDot{.}
\def\@seccntformat#1{\csname the#1\endcsname\@seccntDot\hskip 0.5em}
\makeatother

\journal{Comptes Rendus}
\date{19 juillet 2010}
\begin{document}

\begin{frontmatter}





\selectlanguage{english}
\title{Multiarray Signal Processing: Tensor decomposition meets compressed sensing}


\author[Lim]{Lek-Heng~Lim}
\ead{lekheng@math.berkeley.edu}
\author[Comon]{Pierre Comon}
\ead{pcomon@unice.fr}

\address[Lim]{Department of Mathematics, University of California, Berkeley, CA 94720-3840}
\address[Comon]{Lab.~I3S, CNRS UMR6070, University of Nice, F-06903, Sophia-Antipolis, France}


\begin{abstract}
We discuss how recently discovered techniques and tools from compressed
sensing can be used in tensor decompositions, with a view towards modeling signals from multiple arrays of multiple sensors.
We show that with appropriate bounds on a measure of separation between radiating sources called coherence, one could always guarantee the existence and uniqueness of a best rank-$r$ approximation of the tensor representing the signal.
We also deduce a computationally feasible variant of Kruskal's uniqueness condition, where the coherence appears as a proxy for $k$-rank.
Problems of sparsest recovery with an infinite continuous dictionary, lowest-rank tensor representation, and blind source separation are treated in a uniform fashion. The decomposition of the measurement tensor leads to simultaneous localization and extraction of radiating sources, in an entirely deterministic manner.


\vskip 0.5\baselineskip

\selectlanguage{francais}

\noindent{\bf R\'esum\'e}
\vskip 0.5\baselineskip
\noindent
\textbf{Traitement du signal multi-antenne: les d\'ecompositions tensorielles rejoignent l'\'echantillonnage compress\'e.}  Nous d\'ecrivons comment les techniques et outils d'\'echantillonnage compress\'e r\'ecemment d\'ecouverts peuvent \^etre utilis\'es dans les d\'ecompositions tensorielles, avec pour illustration une mod\'elisation des signaux provenant de plusieurs antennes multicapteurs.
Nous montrons qu'en posant des bornes appropri\'ees sur une certaine mesure de s\'eparation entre les sources rayonnantes (appel\'ee  coh\'erence dans le jargon de l'\'echantillonnage compress\'e), on pouvait toujours garantir l'existence et l'unicit\'e d'une meilleure  approximation de rang $r$ du tenseur repr\'esentant le signal.
Nous en d\'eduisons aussi une variante calculable de la condition d'unicit\'e de Kruskal, o\`u cette coh\'erence appara\^it comme une mesure du $k$-rang. Les probl\`emes de r\'ecup\'eration parcimonieuse avec un dictionnaire infini continu, de repr\'esentation tensorielle de plus bas rang, et de s\'eparation aveugle de sources sont ainsi abord\'es d'une seule et m\^eme fa\c{c}on. La d\'ecomposition du tenseur de mesures conduit \`a la localisation et \`a l'extraction simultan\'ees des sources rayonnantes, de mani\`ere enti\`erement d\'eterministe.


\end{abstract}


\end{frontmatter}

\begin{center}
\small
Submitted December 16, 2009;
Accepted June 29, 2010;
Available online July 29, 2010.\\
Presented by O.~Macchi.
\end{center}

\newpage



\selectlanguage{francais}
\section*{Version fran\c{c}aise abr\'eg\'ee}

Nous expliquons comment les d\'ecompositions tensorielles et les mod\`eles d'approximation apparaissent naturellement dans les signaux multicapteurs, et voyons comment l'\'etude de ces mod\`eles peut \^etre enrichie par des contributions provenant de l'\'echantillonnage compress\'e. Le vocable \textit{\'echantillonnage compress\'e} est \`a prendre au sens large, englobant non seulement les id\'ees couvertes par \cite{C, CR, CDD, D, DE, GN}, mais aussi les travaux sur la minimisation du rang et la compl\'etion de matrice  \cite{CR1,CT, FHB, Gross, RFP, ZSY}.

Nous explorons notamment  deux  th\`emes: (1)~l'utilisation de dictionnaires redondants avec des bornes sur les produits scalaires entre leurs \'el\'ements;  (2)~le recours \`a la coh\'erence ou au spark pour prouver l'unicit\'e. 
En particulier, nous verrons comment ces id\'ees peuvent \^etre \'etendues aux tenseurs, et appliqu\'ees \`a leur d\'ecomposition et leurs approximations.
Si nous qualifions les travaux \cite{C, CR, CDD, D,
DE, GN} d' ``\'echantillonnage compress\'e de formes lin\'eaires'' (variables vectorielles) et \cite{CR1, CT, FHB, Gross, RFP, ZSY} d' ``\'echantillonnage compress\'e de formes bilin\'eaires'' (variables matricielles), alors cet article porte sur l'\'echantillonnage compress\'e de formes multilin\'eaires (variables tensorielles).

Les approximations tensorielles rec\`elent des difficult\'es dues \`a leur caract\`ere mal pos\'e \cite{CGLM, dSL}, et le calcul de la plupart des probl\`emes d'alg\`ebre multilin\'eaire sont de complexit\'e non polynomiale (NP-durs) \cite{Haa, HL1}.
En outre, il est souvent difficile ou m\^eme impossible de r\'epondre dans le cadre de la g\'eom\'etrie alg\'ebrique \`a certaines questions fondamentales concernant les tenseurs, cadre qui est pourtant usuel pour formuler ces questions (cf.\ Section~\ref{sec:Exist}).
Nous verrons que certains de ces probl\`emes pourraient devenir plus abordables si on les d\'eplace de la g\'eom\'erie alg\'ebrique vers l'analyse harmonique. Plus pr\'ecis\'ement, nous verrons comment les concepts glan\'es aupr\`es de l'\'echantillonnage compress\'e peuvent \^etre utilis\'es pour att\'enuer certaines difficult\'es.

Enfin, nous montrons que si les sources sont suffisamment s\'epar\'ees, alors il est possible de les localiser et de les extraire, d'une mani\`ere compl\`etement d\'eterministe.
Par ``suffisamment s\'epar\'ees'', on entend que certains produits scalaires soient inf\'erieurs \`a un seuil, qui diminue avec le nombre de sources pr\'esentes. Dans le jargon de l'\'echantillonnage compress\'e, la ``coh\'erence'' d\'esigne le plus grand de ces produits scalaires.
En posant des bornes appropri\'ees sur cette coh\'erence, on peut toujours garantir l'existence et l'unicit\'e d'une meilleure  approximation de rang $r$ d'un tenseur, et par cons\'equent l'identifiabilit\'e d'un canal de propagation d'une part, et l'estimation des signaux source d'autre part.

\selectlanguage{english}

\section{Introduction\label{sec:Intro}}

We discuss how tensor decomposition and approximation models arise naturally
in multiarray multisensor signal processing and see how the studies of such models are
enriched by mathematical innovations coming from compressed sensing. We
interpret the term \textit{compressed sensing} in a loose and broad sense,
encompassing not only the ideas covered in \cite{C, CR, CDD, D, DE, GN} but
also the line of work on rank minimization and matrix completion in \cite{CR1,
CT, FHB, Gross, RFP, ZSY}. We explore two themes in particular: (1)~the use of
overcomplete dictionaries with bounds on coherence; (2)~the use of spark  or  coherence to obtain uniqueness results. In
particular we will see how these ideas may be extended to tensors and applied
to their decompositions and approximations. If we view \cite{C, CR, CDD, D,
DE, GN} as `compressed sensing of linear forms' (vector variables) and
\cite{CR1, CT, FHB, Gross, RFP, ZSY} as `compressed sensing of bilinear
forms' (matrix variables), then this article is about `compressed sensing of
multilinear forms'  (tensor variables), where these vectors, matrices, or
tensors are signals measured by sensors or arrays of sensors.

Tensor approximations are fraught with ill-posedness difficulties \cite{CGLM,
dSL} and computations of most multilinear algebraic problems are NP-hard
\cite{Haa, HL1}. Furthermore even some of the most basic questions about
tensors are often difficult or even impossible to answer within the framework
of algebraic geometry, the usual context for formulating such questions
(cf.\ Section~\ref{sec:Exist}). We will see that some of these problems with
tensors could become more tractable when we move from algebraic geometry to
slightly different problems within the framework of harmonic analysis. More
specifically we will show how wisdom gleaned from compressed sensing could be
used to alleviate some of these issues.

This article is intended to be a short communication. Any result whose proof
requires more than a few lines of arguments is not mentioned at all but
deferred to our full paper \cite{CL}. Relations with other aspects of compressed sensing
beyond the  two  themes mentioned above, most notably exact recoverability results under the restricted isometry property \cite{C1} or coherence assumptions \cite{Tropp},
are also deferred to \cite{CL}.  While the discussions in this article are limited to order-$3$ tensors, it is entirely straightforward to extend them to tensors of any higher order.

\section{Multisensor signal processing\label{sec:App}}

Tensors are well-known to arise in signal processing as higher order
\textit{cumulants} in independent component analysis \cite{Como04:ijacsp}
and have
been used successfully in blind source separation \cite{ComoJ09}.  The
signal processing application considered here is of a different nature but
also has a natural tensor decomposition model. Unlike the amazing
\textit{single-pixel} camera \cite{Single-pixel} that is celebrated in
compressed sensing, this application comes from the opposite end and involves
\textit{multiple arrays of multiple sensors} \cite{SidiBG00:ieeesp}.

Consider an array of $l$ sensors, each located in space at a position defined
by a vector $\mathbf{b}_{i}\in\mathbb{R}^{3}$, $i=1,\dots,l$. Assume this
array is impinged by $r$ narrow-band waves transmitted by independent
radiating sources through a linear stationary medium. Denote by $\sigma
_{p}(t_{k})$ the complex envelope of the $p$th source, $1\leq p\leq r$, where
$t_{k}$ denotes a point in time and $k=1,\dots,n$. If the location of source $p$
is characterized by a parameter vector $\boldsymbol{\theta}_{p}$, the signal
received by sensor $i$ at time $t_{k}$ can be written as
\begin{equation}
s_{i}(k)=\sum\nolimits_{p=1}^{r}\sigma_{p}(t_{k})\,\varepsilon_{i}%
(\boldsymbol{\theta}_{p}) \label{array-eq}%
\end{equation}
where $\varepsilon_{i}$ characterizes the response of sensor $i$ to external excitations.

Such multisensor arrays occur in a variety of applications  including acoustics,
neuroimaging, and telecommunications.  The sensors could be antennas, EEG
electrodes, microphones, radio telescopes, etc, capturing signals in the form
of images, radio waves, sounds, ultrasounds, etc, emanating from sources that
could be cell phones, distant galaxies, human brain, party conversations, etc.

\begin{example}
For instance, if one considers the transmission of narrowband electromagnetic
waves over air, $\varepsilon_{i}(\boldsymbol{\theta}_{p})$ can be
assimilated to a pure complex exponential (provided the differences between
time delays of arrival are much smaller than the inverse of the bandwidth):
\begin{equation}\label{narrowBand-eq}
\varepsilon_{i}(\boldsymbol{\theta}_{p})\approx\exp(\psi_{i,p}),\quad
\psi_{i,p}:=\imath\frac{\omega}{c}\,\left(  \mathbf{b}_{i}^{\top}%
\mathbf{d}_{p}-\frac{1}{2R_{p}}\,\lVert\mathbf{b}_{i}\wedge\mathbf{d}%
_{p}\rVert_{2}^{2}\right)
\end{equation}
where the $p^{th}$ source location is defined by its direction $\mathbf{d}_{p}\in\mathbb{R}^3$ and distance $R_{p}$ from an arbitrarily chosen origin $O$, $\omega$ denotes
the central pulsation, $c$ the wave celerity, $\imath^{2}=-1$, and $\wedge$ the vector wedge product. More generally, one may consider $\psi_{i,p}$ to be a sum of functions whose
variables separate, i.e.$\ \psi_{i,p}=\mathbf{f}(i)^{\top}\mathbf{g}(p)$,
where $\mathbf{f}(i)$ and $\mathbf{g}(p)$ are vectors of the same dimension.
Note that if sources are in the far field ($R_p\gg1$), then the last term in the expression of $\psi_{i,p}$ in \eqref{narrowBand-eq} may be neglected.
\end{example}

\subsection{Structured multisensor arrays\label{subarrays-sec}}

We are interested in sensor arrays enjoying an invariance property.
We assume that there are $m$ arrays, each having the same number $l$ of sensors. They do not need to be disjoint, that is, two different arrays may share one or more sensors.

From \eqref{array-eq}, the signal received by the $j$th array, $j=1,\dots,m$,
takes the form
\begin{equation}
s_{i,j}(k)=\sum\nolimits_{p=1}^{r}\sigma_{p}(t_{k})\,\varepsilon_{i,j}%
(\boldsymbol{\theta}_{p}). \label{arrays-eq}%
\end{equation}
The invariance property\footnote{So called as the property follows from translation invariance: angles of arrival remain the same. The term was probably first used in \cite{ESPRIT}.} that we are interested in can be expressed as
\begin{equation}
\varepsilon_{i,j}(\boldsymbol{\theta}_{p})=\varepsilon_{i,1}%
(\boldsymbol{\theta}_{p})\,\varphi(j,p). \label{invariance-eq}%
\end{equation}
In other words, variables $i$ and $j$ decouple.

\begin{figure}[tbh]
\centerline{\includegraphics[scale=0.5]{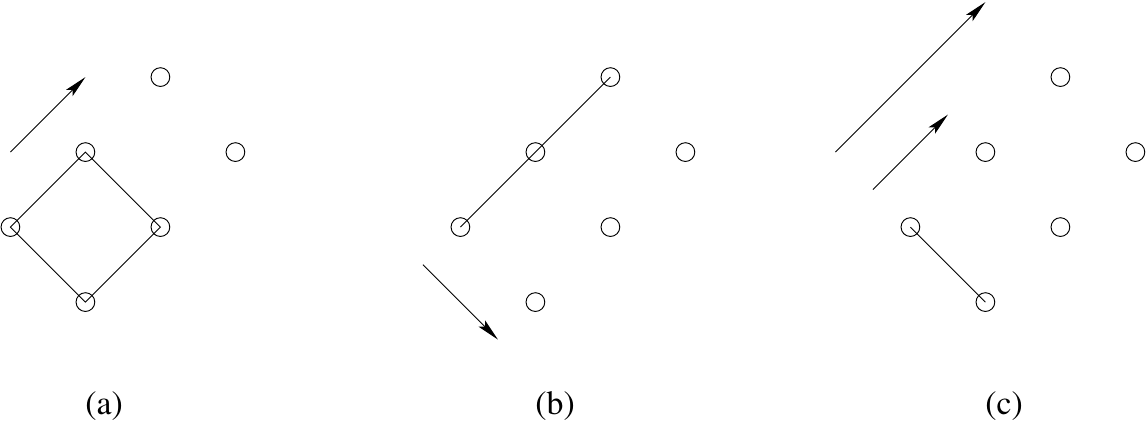}}\caption{From the same array of sensors, several subarrays can be defined that differ from each other by translation:
(a) two (overlapping) subarrays of $4$ sensors, (b) two subarrays of
$3$ sensors, (c) three subarrays of $2$ sensors.}\label{fig:array3sensors}%
\end{figure}

This property is encountered in the case of arrays that can be obtained from
each other by a translation (see Fig. \ref{fig:array3sensors}). Assume sources are in the far field. Denote by $\Delta_{j}$ the
vector that allows deduction of the locations of sensors in the $j$th array from
those of $1$st array. Under these hypotheses,  we have for the
first array, $\psi_{i,p,1}=\imath\frac{\omega}{c}(\mathbf{b}_{i}^{\top
}\mathbf{d}_{p})$. By a translation of $\Delta_{j}$ we obtain the phase
response of the $j$th array as:
\[
\psi_{i,p,j}=\imath\frac{\omega}{c}(\mathbf{b}_{i}^{\top}\mathbf{d}_{p}%
+\Delta_{j}^{\top}\mathbf{d}_{p}).
\]
Observe that indices $i$ and $j$ decouple upon exponentiation and that we
have $\varphi(j,p)=\exp\left(  \imath\frac{\omega}{c}\,\Delta_{j}^{\top
}\mathbf{d}_{p}\right)  $.

Now plug the invariance expression \eqref{invariance-eq} into
\eqref{arrays-eq} to obtain the observation model:
\[
s_{i,j}(k)=\sum\nolimits_{p=1}^{r}\varepsilon_{i,1}(\boldsymbol{\theta}%
_{p})\, \varphi(j,p)\, \sigma_{p}(t_{k}),\quad i=1,\dots,l;\;j=1,\dots
,m;\;k=1,\dots,n.
\]
This simple multilinear model is the one that we shall discuss in this
article. Note that the left hand side is measured, while the quantities on the
right hand side are to be estimated. If we rewrite $a_{ijk}=s_{i,j}(k)$,
$u_{ip}=\hat{\varepsilon}_{i,1}(\boldsymbol{\theta}_{p})$, $v_{jp}%
=\hat{\varphi}(j,p)$, $w_{kp}=\hat{\sigma}_{p}(t_{k})$ (where the `hat'
indicates that the respective quantities are suitably normalized) and
introduce a scalar $\lambda_{p}$ to capture the collective magnitudes, we get
the \textit{tensor decomposition} model%
\[
a_{ijk}=\sum\nolimits_{p=1}^{r}\lambda_{p}\,u_{ip}v_{jp}w_{kp},\quad
i=1,\dots,l;\;j=1,\dots,m;\;k=1,\dots,n,
\]
with $\lVert\mathbf{u}_{p}\rVert=\lVert\mathbf{v}_{p}\rVert=\lVert
\mathbf{w}_{p}\rVert=1$. In the presence of noise, we often seek a
\textit{tensor approximation} model with respect to some measure of nearness,
say, a sum-of-squares loss that is common when the noise is assumed
white and  Gaussian:
\[
\operatornamewithlimits{min}_{\lambda_p, \mathbf{u}_{p},\mathbf{v}_{p},\mathbf{w}_{p}} \sum\nolimits_{i,j,k=1}^{l,m,n}\left[  a_{ijk}-\sum\nolimits_{p=1}^{r}%
\lambda_{p}\, u_{ip}v_{jp}w_{kp}\right]  ^{2}.
\]
Our model has the following physical interpretation: if $a_{ijk}$ is the array of measurements recorded from sensor $i$ of subarray $j$ at time $k$, then it is ideally written as a sum of $r$ individual source contributions $\sum\nolimits_{p=1}^{r}
\lambda_{p}\, u_{ip}\,v_{jp}\,w_{kp}$. Here, $u_{ip}$ represent the transfer functions among sensors of the same subarray, $v_{jp}$ the transfer between subarrays, and $w_{kp}$ the discrete-time source signals.
All these quantities can be identified. In other words, the exact way one subarray can be deduced from the others does not need to be known. Only the existence of this geometrical invariance is required.
\section{Tensor rank\label{sec:Def}}
Let $\mathbb{V}_{1},\dots,\mathbb{V}_{k}$ be vector spaces over  a field, say, $\mathbb{C}$. An element
of the tensor product $\mathbb{V}_{1}\otimes\dots\otimes\mathbb{V}_{k}$ is
called an order-$k$ \textit{tensor} or $k$-\textit{tensor} for short. Scalars,
vectors, and matrices may be regarded as tensors of order $0$, $1$, and
$2$ respectively. For the purpose of this article and for notational
simplicity, we will limit our discussions to $3$-tensors.
Denote by $l, m, n$ the dimensions of $\mathbb{V}_{1}$, $\mathbb{V}_{2}$, and $\mathbb{V}_{3}$, respectively.
Up to a choice of
bases on $\mathbb{V}_{1},\mathbb{V}_{2},\mathbb{V}_{3}$, a $3$-tensor in
$\mathbb{V}_{1}\otimes\mathbb{V}_{2}\otimes\mathbb{V}_{3}$ may be represented
by an $l\times m\times n$ array of elements of $\mathbb{C}$,
\[
A=(a_{ijk})_{i,j,k=1}^{l,m,n}\in\mathbb{C}^{l\times m\times n}.
\]
These are sometimes called \textit{hypermatrices}\footnote{The subscripts and
superscripts will be dropped when the range of $i,j,k$ is obvious or
unimportant.} and come equipped with certain algebraic operations inherited
from the algebraic structure of $\mathbb{V}_{1}\otimes\mathbb{V}_{2}%
\otimes\mathbb{V}_{3}$. The one that interests us most is the decomposition of
$A=(a_{ijk})\in\mathbb{C}^{l\times m\times n}$ as%
\begin{equation}
A=\sum\nolimits_{p=1}^{r}\lambda_{p}\,\mathbf{u}_{p}\otimes\mathbf{v}%
_{p}\otimes\mathbf{w}_{p},\qquad a_{ijk}=\sum\nolimits_{p=1}^{r}\lambda
_{p}\, u_{ip}v_{jp}w_{kp},\label{prop2a}%
\end{equation}
with $\lambda_{p}\in\mathbb{C}$, $\mathbf{u}_{p}\in\mathbb{C}^{l}%
,\mathbf{v}_{p}\in\mathbb{C}^{m},\mathbf{w}_{p}\in\mathbb{C}^{n}$. For
$\mathbf{u}=[u_{1},\dots,u_{l}]^{\top}$, $\mathbf{v}=[v_{1},\dots,v_{m}%
]^{\top}$, $\mathbf{w}=[w_{1},\dots,w_{n}]^{\top}$, we write $\mathbf{u}%
\otimes\mathbf{v}\otimes\mathbf{w}:=(u_{i}v_{j}w_{k})_{i,j,k=1}^{l,m,n}%
\in\mathbb{C}^{l\times m\times n}$. This generalizes $\mathbf{u}%
\otimes\mathbf{v}=\mathbf{uv}^{\top}$ in the case of matrices.

A different choice of bases on $\mathbb{V}_{1},\dots,\mathbb{V}_{k}$ would
lead to a different hypermatrix representation of elements in $\mathbb{V}%
_{1}\otimes\dots\otimes\mathbb{V}_{k}$. For the more pedantic readers, it is
understood that what we call a tensor in this article really means a
hypermatrix. The decomposition of a tensor into a linear combination of
rank-$1$ tensors was first studied in \cite{Hi1}.

\begin{definition}
A tensor that can be expressed as an outer product of vectors is called
\textbf{decomposable} (or \textbf{rank-one} if it is also nonzero). More
generally, the \textbf{rank} of a tensor $A=(a_{ijk})_{i,j,k=1}^{l,m,n}%
\in\mathbb{C}^{l\times m\times n}$, denoted $\operatorname*{rank}(A)$, is
defined as the minimum $r$ for which $A$ may be expressed as a sum of $r$
rank-$1$ tensors,%
\begin{equation}
\operatorname*{rank}(A):=\min\Bigl\{r\Bigm|A=\sum\nolimits_{p=1}^{r}%
\lambda_{p}\,\mathbf{u}_{p}\otimes\mathbf{v}_{p}\otimes\mathbf{w}_{p}\Bigr\}.
\label{cp2}%
\end{equation}

\end{definition}

We will call a decomposition of the form \eqref{prop2a} a
\textit{rank-revealing decomposition} when $r=\operatorname*{rank}(A)$. The
definition of rank in \eqref{cp2} agrees with matrix rank when applied to an
order-$2$ tensor.

$\mathbb{C}^{l\times m\times n}$ is a Hilbert space of dimension $lmn$, equipped  with the \textit{Frobenius} (or \textit{Hilbert-Schmidt}) norm, and its associated scalar product:
\[
\lVert A\rVert_{F}=\left[  \sum\nolimits_{i,j,k=1}%
^{l,m,n}\lvert a_{ijk}\rvert^{2}\right]  ^{\frac{1}{2}},\qquad \langle A,B\rangle_{F}=\sum\nolimits_{i,j,k=1}^{l,m,n}a_{ijk}\overline
{b}_{ijk}.
\]
One may also define tensor norms that are the $\ell^{p}$ equivalent of
Frobenius norm \cite{LC} and tensor norms that are analogous to operator norms of matrices \cite{HL1}.

\section{Existence\label{sec:Exist}}

The problem that we consider here is closely related to the best $r$-term
approximation problem in nonlinear approximations, with one notable difference
--- our dictionary is a continuous manifold, as opposed to a discrete set, of
atoms. We approximate a general signal $\mathbf{v}\in\mathbb{H}$ with an
$r$-term approximant over some \textit{dictionary} of atoms $\mathcal{D}$,
i.e.\ $\mathcal{D}\subseteq\mathbb{H}$ and $\overline{\operatorname*{span}%
(\mathcal{D})}=\mathbb{H}$. We refer the reader to \cite{CDD} for a discussion of the connection between compressed sensing and nonlinear
approximations. We denote the set of $r$-term aproximants by $\Sigma
_{r}(\mathcal{D}):=\{\lambda_{1}\mathbf{v}_{1}+\dots+\lambda_{r}\mathbf{v}%
_{r}\in\mathbb{H}\mid\mathbf{v}_{1},\dots,\mathbf{v}_{r}\in\mathcal{D}%
,\;\lambda_{1},\dots,\lambda_{r}\in\mathbb{C}\}$. Usually $\mathcal{D}$ is
finite or countable but we have a continuum of atoms comprising all
decomposable tensors. The set of decomposable tensors
\[
\operatorname*{Seg}(l,m,n):=\{A\in\mathbb{C}^{l\times m\times n}%
\mid\operatorname*{rank}(A)\leq1\}=\{\mathbf{x}\otimes\mathbf{y}%
\otimes\mathbf{z}\mid\mathbf{x}\in\mathbb{C}^{l},\;\mathbf{y}\in\mathbb{C}%
^{m},\;\mathbf{z}\in\mathbb{C}^{n}\}
\]
is known in geometry as the \textit{Segre variety}. It has the structure of
both a smooth manifold and an algebraic variety, with dimension $l+m+n$
(whereas finite or countable dictionaries are $0$-dimensional). The set of
$r$-term approximants in our case is the $r$th \textit{secant quasiprojective
variety} of the Segre variety, $\Sigma_{r}(\operatorname*{Seg}(l,m,n))=\{A\in
\mathbb{C}^{l\times m\times n}\mid\operatorname*{rank}(A)\leq r\}$. Such a set
may not be closed nor irreducible. In order to study this set using standard
tools of algebraic geometry \cite{CGG2, LM, Z}, one often considers a simpler
variant called the $r$th \textit{secant variety} of the Segre variety, the
(Zariski) closure of $\Sigma_{r}(\operatorname*{Seg}(l,m,n))$. Even with this
simplification, many basic questions remain challenging and open: For example,
it is not known what the value of the generic rank\footnote{Roughly speaking,
this is the value of $r$ such that a randomly generated tensor will have rank
$r$. For $m\times n$ matrices, the generic rank is $\min\{m,n\}$ but $l\times
m\times n$ tensors in general have generic rank $>\min\{l,m,n\}$.} is for
general values of $l,m,n$ \cite{CGG2}; nor are the polynomial
equations\footnote{For matrices, these equations are simply given by the
vanishing of the $k\times k$ minors for all $k>r$.} defining the $r$th secant
variety known in general \cite{LM}.

The seemingly innocent remark in the preceding paragraph that for $r>1$, the
set $\{A\in\mathbb{C}^{l\times m\times n}\mid\operatorname*{rank}(A)\leq r\}$
is in general not a closed set has implication on the model that we proposed.
Another way to view this is that tensor rank for tensors of order $3$ or
higher is not an upper semicontinuous function \cite{dSL}. Note that tensor
rank for order-$2$ tensors (i.e.\ matrix rank) is upper semicontinuous: if
$A$ is a matrix and $\operatorname*{rank}(A)=r$, then $\operatorname*{rank}%
(B)\geq r$ for all matrices~$B$ in a sufficiently small neighborhood of~$A$. As a consequence,
the best rank-$r$ approximation problem for tensors,
\begin{equation}
\operatornamewithlimits{argmin}_{\lVert\mathbf{u}_{p}\rVert_{2}=\lVert\mathbf{v}%
_{p}\rVert_{2}=\lVert\mathbf{w}_{p}\rVert_{2}=1}\lVert A-\lambda_{1}%
\mathbf{u}_{1}\otimes\mathbf{v}_{1}\otimes\mathbf{w}_{1}-\dots-\lambda
_{r}\mathbf{u}_{r}\otimes\mathbf{v}_{r}\otimes\mathbf{w}_{r}\rVert_{F},
\label{eq:approx}
\end{equation}
unlike that for matrices, does not in general have a solution. The following is a simple example taken from \cite{dSL}.

\begin{example}\label{eg:ill}
Let $\mathbf{u}_{i},\mathbf{v}_{i}\in\mathbb{C}^{m}$, $i=1,2,3$. Let
$A:=\mathbf{u}_{1}\otimes\mathbf{u}_{2}\otimes\mathbf{v}_{3}+\mathbf{u}%
_{1}\otimes\mathbf{v}_{2}\otimes\mathbf{u}_{3}+\mathbf{v}_{1}\otimes
\mathbf{u}_{2}\otimes\mathbf{u}_{3}$ and for $n\in\mathbb{N}$, let
\[
A_{n}:=n\left(  \mathbf{u}_{1}+\frac{1}{n}\mathbf{v}_{1}\right)
\otimes\left(  \mathbf{u}_{2}+\frac{1}{n}\mathbf{v}_{2}\right)  \otimes\left(
\mathbf{u}_{3}+\frac{1}{n}\mathbf{v}_{3}\right)  -n\mathbf{u}_{1}%
\otimes\mathbf{u}_{2}\otimes\mathbf{u}_{3}.
\]
One may show that $\operatorname*{rank}(A)=3$ iff $\mathbf{u}_{i}%
,\mathbf{v}_{i}$ are linearly independent, $i=1,2,3$. Since it is clear that
$\operatorname*{rank}(A_{n})\leq2$ by construction and $\lim_{n\rightarrow\infty}A_{n}=A$, the
rank-$3$ tensor $A$ has no best rank-$2$ approximation. Such a tensor is said
to have border rank $2$.
\end{example}

This phenomenon where a tensor fails to have a best rank-$r$ approximation is
much more widespread than one might imagine, occurring over a wide range of
dimensions, orders, and ranks; happens regardless of the choice of norm (or
even Br\`{e}gman divergence) used. These counterexamples occur with positive
probability and in some cases with certainty (in $\mathbb{R}^{2\times2\times
2}$ and $\mathbb{C}^{2\times2\times2}$, no tensor of rank-$3$ has a best
rank-$2$ approximation). We refer the reader to \cite{dSL} for further details.

Why not consider approximation by tensors in the closure of the set of all
rank-$r$ tensors, i.e.\ the $r$th secant variety, instead? Indeed this was the
idea behind the \textit{weak solutions} suggested in \cite{dSL}. The trouble
with this approach is that it is not known how one could parameterize
the $r$th secant variety in general: While we know that all elements of
the $r$th secant quasiprojective variety $\Sigma_{r}(\operatorname*{Seg}(l,m,n))$
may be parameterized as $\lambda_{1}\mathbf{u}_{1}\otimes\mathbf{v}_{1}%
\otimes\mathbf{w}_{1}+\dots+\lambda_{r}\mathbf{u}_{r}\otimes\mathbf{v}%
_{r}\otimes\mathbf{w}_{r}$, it is not known how one could parameterize the
limits of these, i.e.\ the additional elements that occur in the closure of $\Sigma_{r}(\operatorname*{Seg}(l,m,n))$,
when $r > \operatorname*{min}\{l,m,n\}$.  More specifically,
if $r \le \operatorname*{min}\{l,m,n\}$, Terracini's Lemma
\cite{Z} provides a way to do this since generically a rank-$r$ tensor has the form 
$\lambda_{1}\mathbf{u}_{1}\otimes\mathbf{v}_{1}%
\otimes\mathbf{w}_{1}+\dots+\lambda_{r}\mathbf{u}_{r}\otimes\mathbf{v}%
_{r}\otimes\mathbf{w}_{r}$ where $\{\mathbf{u}_{1},\dots,\mathbf{u}_{r}\}$, $\{\mathbf{v}_{1},\dots,\mathbf{v}_{r}\}$, $\{\mathbf{w}_{1},\dots,\mathbf{w}_{r}\}$ are linearly independent; but when $r > \operatorname*{min}\{l,m,n\}$, this generic linear independence does not hold and there are no known ways to parameterize a rank-$r$ tensor in this case. 

We propose that a better way would be to introduce natural
\textit{a priori} conditions that prevent the phenomenon in Example~\ref{eg:ill} from occurring. An example of such conditions is nonnegativity restrictions on $\lambda_i, \mathbf{u}_i, \mathbf{v}_i$, examined in our earlier work \cite{LC}. Here we will impose much weaker and more natural restrictions motivated by the notion of \textit{coherence}.
Recall that a real valued function $f$ with an unbounded domain $\operatorname*{dom}(f)$
and $\lim_{\mathbf{x}\in\operatorname*{dom}(f),\;\lVert\mathbf{x}%
\rVert\rightarrow+\infty}f(\mathbf{x})=+\infty$ is called \textit{coercive}
(or $0$\textit{-coercive}) \cite{HL}. A nice feature of such functions is that
the existence of a global minimizer is guaranteed. The objective function in
\eqref{eq:approx} is not coercive in general but we will show here that a mild
condition on coherence, a notion that frequently appears in recent
work on compressed sensing, allows us to obtain a coercive function and therefore circumvent the non-existence difficulty. In the context of our application in Section~\ref{sec:App}, coherence quantifies the minimal angular separation in space or the maximal cross correlation in time between the radiating sources. \begin{definition}
Let $\mathbb{H}$ be a Hilbert space and $\mathbf{v}_{1},\dots,\mathbf{v}%
_{r}\in\mathbb{H}$ be a finite collection of unit vectors, i.e.\ $\lVert
\mathbf{v}_{p}\rVert_{\mathbb{H}}=1$. The \textbf{coherence} of the collection
$V=\{\mathbf{v}_{1},\dots,\mathbf{v}_{r}\}$ is defined as $\mu(V):=\max_{p\neq
q}\lvert\langle\mathbf{v}_{p},\mathbf{v}_{q}\rangle\rvert$.
\end{definition}

This notion has been introduced in slightly different forms and names: mutual
incoherence of two dictionaries \cite{DE}, mutual coherence of two
dictionaries \cite{CR}, the coherence of a subspace projection \cite{CT}, etc.
The version here follows that of \cite{GN}. We will be interested in the
case when $\mathbb{H}$ is finite dimensional (in particular $\mathbb{H}%
=\mathbb{C}^{l\times m\times n}$ or $\mathbb{C}^{m}$). When $\mathbb{H}%
=\mathbb{C}^{m}$, we often regard $V$ as an $m\times r$ matrix whose column
vectors are $\mathbf{v}_{1},\dots,\mathbf{v}_{r}$. Clearly $0\leq\mu(V)\leq1$,
$\mu(V)=0$ iff $\mathbf{v}_{1},\dots,\mathbf{v}_{r}$ are orthonormal, and
$\mu(V)=1$ iff $V$ contains at least a pair of collinear  vectors.

While a solution to the best rank-$r$ approximation problem \eqref{eq:approx}
may not exist, the following shows that a solution to the bounded coherence
best rank-$r$ approximation problem \eqref{eq:bdapprox} always exists.

\begin{theorem}
\label{thm:Exist-mu}Let $A\in\mathbb{C}^{l\times m\times n}$ and let
$\mathcal{U}=\{U\in\mathbb{C}^{l\times r}\mid\mu(U)\leq\mu_{1}\}$,
$\mathcal{V}=\{V\in\mathbb{C}^{m\times r}\mid\mu(V)\leq\mu_{2}\}$,
$\mathcal{W}=\{W\in\mathbb{C}^{n\times r}\mid\mu(W)\leq\mu_{3}\}$, be families
of dictionaries of unit vectors of coherence not more than $\mu_{1},\mu
_{2},\mu_{3}$ respectively. If
\[
\mu_{1}\mu_{2}\mu_{3}<\frac{1}{r},
\]
then the infimum $\eta$ defined as \begin{equation}
\eta = \inf\biggl\{\left\Vert A-\sum\nolimits_{p=1}^{r}\,\lambda_{p}\mathbf{u}%
_{p}\otimes\mathbf{v}_{p}\otimes\mathbf{w}_{p}\right\Vert
\biggm|\boldsymbol{\lambda}\in\mathbb{C}^{r},U\in\mathcal{U},V\in
\mathcal{V},W\in\mathcal{W}\biggr\} \label{eq:bdapprox}%
\end{equation}
is attained. Here $\lVert\,\cdot\,\rVert$ denotes any norm on $\mathbb{C}%
^{l\times m\times n}$.
\end{theorem}

\begin{proof}
Since all norms are equivalent on a finite dimensional space, we may assume that  $\lVert\,\cdot\,\rVert=\lVert\,\cdot\,\rVert_{F}$, the Frobenius norm.  Let the objective function $f:\mathbb{C}^{r}\times\mathcal{U}\times\mathcal{V}%
\times\mathcal{W}\rightarrow[0,\infty)$ be%
\begin{equation}
f(\boldsymbol{\lambda},U,V,W):=\left\Vert A-\sum\nolimits_{p=1}^{r}\lambda
_{p}\mathbf{u}_{p}\otimes\mathbf{v}_{p}\otimes\mathbf{w}_{p}\right\Vert
_{F}^{2}. \label{eq:f}%
\end{equation}
Let $\mathcal{E}=\mathbb{C}^{r}\times\mathcal{U}\times\mathcal{V}%
\times\mathcal{W}$. Note that $\mathcal{E}$ as a subset of $\mathbb{C}%
^{r(1+l+m+n)}$ is noncompact (closed but unbounded). We write
$T=(\boldsymbol{\lambda},U,V,W)$ and let the infimum in question be $\eta
:=\inf\{f(T)\mid T\in\mathcal{E}\}$. We will show that the sublevel set of $f$
restricted to $\mathcal{E}$, defined as $\mathcal{E}_{\alpha}=\{T\in\mathcal{E}\mid
f(T)\leq\alpha\}$, is compact for all $\alpha>\eta$ and thus the infimum of $f$
on $\mathcal{E}$ is attained. The set $\mathcal{E}_{\alpha}=\mathcal{E}\cap
f^{-1}(-\infty,\alpha]$ is closed since $\mathcal{E}$ is closed and $f$ is continuous (by the continuity
of norm). It remains to show that $\mathcal{E}_{\alpha}$ is bounded. Suppose
the contrary. Then there exists a sequence $(T_{k})_{k=1}^{\infty}%
\subset\mathcal{E}$ with $\lVert T_{k}\rVert_{2}\rightarrow\infty$ but
$f(T_{k})\leq\alpha$ for all $k$. Clearly, $\lVert T_{k}\rVert_{2}%
\rightarrow\infty$ implies that $\lVert\boldsymbol{\lambda}^{(k)}\rVert
_{2}\rightarrow\infty$. Note that%
\[
f(T)\geq\left[  \lVert A\rVert_{F}-\left\Vert \sum\nolimits_{p=1}^{r}%
\lambda_{p}\mathbf{u}_{p}\otimes\mathbf{v}_{p}\otimes\mathbf{w}_{p}\right\Vert
_{F}\right]  ^{2}.
\]
We have%
\begin{align*}
\left\Vert \sum\nolimits_{p=1}^{r}\lambda_{p}\mathbf{u}_{p}\otimes
\mathbf{v}_{p}\otimes\mathbf{w}_{p}\right\Vert _{F}^{2}  &  =\sum
\nolimits_{p,q=1}^{r}\lambda_{p}\overline{\lambda}_{q}\langle\mathbf{u}%
_{p},\mathbf{u}_{q}\rangle\langle\mathbf{v}_{p},\mathbf{v}_{q}\rangle
\langle\mathbf{w}_{p},\mathbf{w}_{q}\rangle\\
&  \geq\sum\nolimits_{p=1}^{r}\lvert\lambda_{p}\rvert^{2}\lVert\mathbf{u}%
_{p}\rVert_{2}^{2}\lVert\mathbf{v}_{p}\rVert_{2}^{2}\lVert\mathbf{w}_{p}%
\rVert_{2}^{2}-\sum\nolimits_{p\neq q}\lvert\lambda_{p}\overline{\lambda}%
_{q}\langle\mathbf{u}_{p},\mathbf{u}_{q}\rangle\langle\mathbf{v}%
_{p},\mathbf{v}_{q}\rangle\langle\mathbf{w}_{p},\mathbf{w}_{q}\rangle\rvert\\
&  \geq\sum\nolimits_{p=1}^{r}\lvert\lambda_{p}\rvert^{2}-\mu_{1}\mu_{2}%
\mu_{3}\left(  \sum\nolimits_{p\neq q}\lvert\lambda_{p}\overline{\lambda}%
_{q}\rvert\right) \\
&  \geq\lVert\boldsymbol{\lambda}\rVert_{2}^{2}-\mu_{1}\mu_{2}\mu_{3}%
\lVert\boldsymbol{\lambda}\rVert_{1}^{2}\geq(1-
r \mu_{1}\mu_{2}\mu_{3}%
)\lVert\boldsymbol{\lambda}\rVert_{2}^{2}.
\end{align*}
The last inequality follows from $\lVert \boldsymbol{\lambda} \rVert_1 \le \sqrt{r} \lVert \boldsymbol{\lambda} \rVert_2$ for any $\boldsymbol{\lambda} \in \mathbb{C}^r$. By our assumption $1-r\mu_{1}\mu_{2}\mu_{3}>0$ and so as $\lVert
\boldsymbol{\lambda}^{(k)}\rVert_{2}\rightarrow\infty$, $f(T_{k}%
)\rightarrow\infty$, which contradicts  the assumption that $f(T_{k})\leq\alpha$
for all $k$.
\end{proof}

\section{Uniqueness\label{sec:Unique}}

While never formally stated, one of the main maxims in compressed sensing is
that `uniqueness implies sparsity'. For example, this is implicit in various
sparsest recovery arguments in \cite{CR, DE, GN} where, depending on context,
`sparsest' may also mean `lowest rank'. We state a simple formulation of this
observation for our purpose. Let $\mathcal{D}$ be a dictionary of atoms in a
vector space $\mathbb{V}$ (over an infinite field). We do not require
$\mathcal{D}$ to be finite or countable. In almost all cases of interest
$\mathcal{D}$ will be overcomplete with high redundancy. For $\mathbf{x}%
\in\mathbb{V}$ and $r\in\mathbb{N}$, by a $\mathcal{D}$-representation, we
shall mean a representation of the form $\mathbf{x}=\alpha_{1}\mathbf{x}%
_{1}+\dots+\alpha_{r}\mathbf{x}_{r}$ where $\mathbf{x}_{1},\dots
,\mathbf{x}_{r}\in\mathcal{D}$ and $\alpha_{1}\cdots\alpha_{r}\neq0$
($\mathbf{x}_{1},\dots,\mathbf{x}_{r}$ are not required to be distinct).

\begin{lemma}
\label{lem:UniSpa}Let $\mathbf{x}=\alpha_{1}\mathbf{x}_{1}+\dots+\alpha
_{r}\mathbf{x}_{r}$ be a $\mathcal{D}$-representation. (i) If this is the
unique $\mathcal{D}$-representation with $r$ terms, then $\mathbf{x}_{1}%
,\dots,\mathbf{x}_{r}$ must be linearly independent. (ii) If this is the
sparsest $\mathcal{D}$-representation, then $\mathbf{x}_{1},\dots
,\mathbf{x}_{r}$ must be linearly independent. (iii) If this is unique, then
it must also be sparsest.
\end{lemma}

\begin{proof}
Suppose $\beta_{1}\mathbf{x}_{1}+\dots+\beta_{r}\mathbf{x}_{r}=\mathbf{0}$ is
a nontrivial linear relation.
(i): Since not all $\beta_{i}$ are $0$ while all
$\alpha_{i}\neq0$, for some $\theta\neq0$ we must have $(\alpha_{1}%
+\theta\beta_{1})\cdots(\alpha_{r}+\theta\beta_{r})\neq0$, which yields a
different $\mathcal{D}$-representation $\mathbf{x}=\mathbf{x}+\theta
\,\mathbf{0}=(\alpha_{1}+\theta\beta_{1})\mathbf{x}_{1}+\dots+(\alpha_{r}%
+\theta\beta_{r})\mathbf{x}_{r}$.
(ii): Say $\beta_{r}\neq0$, then
$\mathbf{x}=(\alpha_{1}-\beta_{r}^{-1}\beta_{1})\mathbf{x}_{1}+\dots
+(\alpha_{r-1}-\beta_{r}^{-1}\beta_{r-1})\mathbf{x}_{r-1}$ is a sparser
$\mathcal{D}$-representation.
(iii): Let $\mathbf{x}=\gamma_{1}\mathbf{y}%
_{1}+\dots+\gamma_{s}\mathbf{y}_{s}$ be a $\mathcal{D}$-representation with
$s<r$. Write $\mathbf{y}_{1}=\sum_{k=1}^{r-s+1}\theta_k\mathbf{y}_{1}$ with $\sum_{k=1}^{r-s+1}\theta_k=1$. Then we obtain an
$r$-term $\mathcal{D}$-representation $\sum_{k=1}^{r-s+1}\gamma_1\,\theta_k\,\mathbf{y}_{1}+\sum_{i=2}^s \gamma_i\,\mathbf{y}_i$, different from the given one. They are different since $\mathbf{y}_1,\,\mathbf{y}_1,\dots,\,\mathbf{y}_1,\,\mathbf{y}_2,\,\dots, \mathbf{y}_s$ are linearly dependent,  whereas (i) implies that $\mathbf{x}_{1},\dots,\mathbf{x}_{r}$ are linearly independent.
\end{proof}

We will now discuss a combinatorial notion useful in guaranteeing uniqueness
or sparsity of $\mathcal{D}$-representations. The notion of the \textit{girth}
of a circuit \cite{Oxley} is standard and well-known in graphical matroids ---
it is simply the length of a shortest cycle of a graph. However the girth of a
circuit in vector matroids, i.e.\ the cardinality of the smallest linearly
dependent subset of a collection of vectors in a vector space, has rarely
appeared in linear algebra. This has led to it being reinvented multiple times
under different names, most notably as \textit{Kruskal rank} or $k$%
\textit{-rank} in tensor decompositions \cite{Krus}, as \textit{spark} in
compressed sensing \cite{DE}, and as $k$\textit{-stability} in coding theory
\cite{ZSY}. The notions of girth, spark, $k$-rank, and $k$-stability
\cite{ZSY} are related as follows.

\begin{lemma}
Let $\mathbb{V}$ be a vector space over a field $\mathbb{F}$ and
$X=\{\mathbf{x}_{1},\dots,\mathbf{x}_{n}\}$ be a finite subset of $\mathbb{V}%
$. Then%
\[
\operatorname*{girth}(X)=\operatorname*{spark}(X)=\operatorname*{krank}(X)+1
\]
and furthermore $X$ is $k$-stable iff $\operatorname*{krank}(X)=n-k$.
\end{lemma}

\begin{proof}
These follow directly from the respective definitions.
\end{proof}

These notions are unfortunately expected to be difficult to compute because of
the following result \cite{Var}.

\begin{theorem}
[Vardy]It is NP-hard to compute the girth of a vector matroid over a finite
field of two elements, $\mathbb{F}_{2}$.
\end{theorem}

A consequence is that spark, $k$-rank, $k$-stability are all NP-hard if the field is $\mathbb{F}_{2}$.
We note here that several authors have assumed
that spark is NP-hard to compute over $\mathbb{F}=\mathbb{R}$ or $\mathbb{C}$
(assuming $\mathbb{Q}$ or $\mathbb{Q}[\imath]$ inputs) but this is actually
unknown. In particular it does not follow from \cite{Nata}. While it is clear that computing spark via a naive exhaustive search has complexity $O(2^{n})$, one may perhaps do better with cleverer algorithms
when $\mathbb{F}=\mathbb{R}$ or $\mathbb{C}$; in fact it is unknown in this
case whether the corresponding decision problem (Given finite $X\subseteq
\mathbb{V}$ and $s\in\mathbb{N}$, is $\operatorname*{spark}(X)=s$?) is
NP-hard. On the other hand it is easy to compute coherence. Even a
straightforward search for an off-diagonal entry of $X^{\top}X$ of maximum
magnitude is of polynomial complexity.  An important observation of \cite{DE} is that
coherence may sometimes be used in place of spark.

One of the early results in compressed sensing \cite{DE, GN} on the uniqueness
of the sparsest solution is that if%
\begin{equation}
\frac{1}{2}\operatorname*{spark}(X)\geq\lVert \boldsymbol{\beta}\rVert_{0} = \operatorname{card}\{\beta_i\neq0\},  \label{eq:Spark}%
\end{equation}
then $\boldsymbol{\beta}\in\mathbb{C}^{n}$ is a unique solution to
$\min\{\lVert \boldsymbol{\beta}\rVert_{0}\mid X \boldsymbol{\beta}=\mathbf{x}\}$.

For readers familiar with
Kruskal's condition that guarantees the uniqueness of tensor decomposition,
the parallel with \eqref{eq:Spark} is hard to miss once we rewrite Kruskal's
condition in the form%
\begin{equation}
\frac{1}{2}[\operatorname*{krank}(X)+\operatorname*{krank}%
(Y)+\operatorname*{krank}(Z)] \PC \ge \operatorname*{rank}(A)+1. \CP \label{eq:Krus}%
\end{equation}
We state a slight variant of Kruskal's result \cite{Krus} here. Note that the scaling ambiguity is unavoidable because of the multilinearity of $\otimes$.
\begin{theorem}
[Kruskal]\label{thm:Krus}If $A=\sum\nolimits_{p=1}^{r}\mathbf{x}_{p}%
\otimes\mathbf{y}_{p}\otimes\mathbf{z}_{p}$ and $\operatorname*{krank}%
(X)+\operatorname*{krank}(Y)+\operatorname*{krank}(Z) \PC \ge 2(r+1) \CP$, then
$r=\operatorname*{rank}(A)$ and the decomposition is unique up to scaling of
the form $\alpha\mathbf{x}\otimes\beta\mathbf{y}\otimes\gamma\mathbf{z}%
=\mathbf{x}\otimes\mathbf{y}\otimes\mathbf{z}$ for $\alpha,\beta,\gamma
\in\mathbb{C}$ with $\alpha\beta\gamma=1$. This inequality is also sharp in the sense that $2r + 2$ cannot be replaced by $2r + 1$.
\end{theorem}

\begin{proof}
The uniqueness was Kruskal's original result in \cite{Krus}; alternate shorter proofs may be found in \cite{Land, Rhodes, SS}. That $r=\operatorname*{rank}%
(A)$ then follows from Lemma~\ref{lem:UniSpa}(iii). The sharpness of the inequality is due to \cite{Derk}.
\end{proof}

Since spark is expected to be difficult to compute, one may substitute
coherence to get a condition \cite{DE, GN} that is easier to check%
\begin{equation}
\frac{1}{2}\left[  1+\frac{1}{\mu(X)}\right]  \geq\lVert \boldsymbol{\beta}\rVert_{0}.
\label{eqGN1}%
\end{equation}
The equation \eqref{eqGN1} relaxes \eqref{eq:Spark} because of the following
result of \cite{DE, GN}.

\begin{lemma}
\label{lem:spark-mu}Let $\mathbb{H}$ be a Hilbert space and $V=\{\mathbf{v}%
_{1},\dots,\mathbf{v}_{r}\}$ be a finite collection of unit vectors in
$\mathbb{H}$. Then%
\[
\operatorname*{spark}(V)\geq1+\frac{1}{\mu(V)}\qquad\text{and}\qquad
\operatorname*{krank}(V)\geq\frac{1}{\mu(V)}.
\]
\end{lemma}

\begin{proof}
Let $\operatorname*{spark}(V)=s=\operatorname*{krank}(V)+1$. Assume without
loss of generality that $\{\mathbf{v}_{1},\dots,\mathbf{v}_{s}\}$ is a minimal
circuit of $V$ and that $\alpha_{1}\mathbf{v}_{1}+\dots+\alpha_{s}%
\mathbf{v}_{s}=0$ with $\lvert\alpha_{1}\rvert=\max\{\lvert\alpha_{1}%
\rvert,\dots,\lvert\alpha_{s}\rvert\}>0$. Taking inner product with
$\mathbf{v}_{1}$ we get $\alpha_{1}=-\alpha_{2}\langle\mathbf{v}%
_{2},\mathbf{v}_{1}\rangle-\dots-\alpha_{s}\langle\mathbf{v}_{s}%
,\mathbf{v}_{1}\rangle$ and so $\lvert\alpha_{1}\rvert\leq(\lvert\alpha
_{2}\rvert+\dots+\lvert\alpha_{s}\rvert)\mu(V)$. Dividing by $\lvert\alpha
_{1}\rvert$\ then yields $1\leq(s-1)\mu(V)$.
\end{proof}

We now characterize the uniqueness of tensor decompositions in terms of
coherence. Note that $\mathbb{C}$ may be replaced by $\mathbb{R}$. By
``unimodulus scaling'', we mean  scaling of the form
$e^{i\theta_{1}}\mathbf{u}\otimes e^{i\theta_{2}}\mathbf{v}\otimes
e^{i\theta_{3}}\mathbf{w}$ where $\theta_{1}+\theta_{2}+\theta_{3}%
\equiv0\operatorname{mod}2\pi$.

\begin{theorem}
\label{thm:Krus-mu}Let $A\in\mathbb{C}^{l\times m\times n}$ and $A=\sum
\nolimits_{p=1}^{r}\,\lambda_{p}\mathbf{u}_{p}\otimes\mathbf{v}_{p}%
\otimes\mathbf{w}_{p}$, where $\lambda_{p}\in\mathbb{C}$,
$\lambda_{p}\neq0$,  and $\lVert
\mathbf{u}_{p}\rVert_{2}=\lVert\mathbf{v}_{p}\rVert_{2}=\lVert\mathbf{w}%
_{p}\rVert_{2}=1$ for all $p=1,\dots,r$. We write $U=\{\mathbf{u}_{1}%
,\dots,\mathbf{u}_{r}\}$, $V=\{\mathbf{v}_{1},\dots,\mathbf{v}_{r}\}$,
$W=\{\mathbf{w}_{1},\dots,\mathbf{w}_{r}\}$. If%
\begin{equation}
\frac{1}{2}\left[  \frac{1}{\mu(U)}+\frac{1}{\mu(V)}+\frac{1}{\mu(W)}\right]
\PC \ge r+1, \CP \label{eq:Krus-mu}%
\end{equation}
then $r=\operatorname*{rank}(A)$ and the rank revealing decomposition is
unique up to unimodulus  scaling.
\end{theorem}

\begin{proof}
If \eqref{eq:Krus-mu} is satisfied, then Kruskal's condition for uniqueness
\eqref{eq:Krus} must also be satisfied by Lemma~\ref{lem:spark-mu}.
\end{proof}

Note that unlike the $k$-ranks in \eqref{eq:Krus}, the coherences in
\eqref{eq:Krus-mu} are trivial to compute. In addition to uniqueness, an easy
but important consequence of Theorem~\ref{thm:Krus-mu} is that it provides a
readily checkable sufficient condition for tensor rank, which is NP-hard over
any field \cite{Haa, HL1}.

\section{Conclusion}\label{sec:Conclude}

The following existence and uniqueness result may be deduced from
Theorems~\ref{thm:Exist-mu} and \ref{thm:Krus-mu}.

\begin{corollary}\label{cor:final}
Let $A\in\mathbb{C}^{l\times m\times n}$. If $\mu_{1},\mu_{2},\mu_{3}%
\in(0,\infty)$ satisfy
\begin{equation}
\frac{1}{\sqrt[3]{\mu_{1}\mu_{2}\mu_{3}}} \PC \ge\frac{2}{3}(r+1), \CP
\end{equation}
then the bounded coherence rank-$r$ approximation problem
\eqref{eq:bdapprox} has a solution that is unique up to unimodulus scaling.
\end{corollary}

\begin{proof}
The case $r=1$ is trivial. For $r\geq2$, since $\mu_{1}\mu_{2}\mu
_{3} \PC \le [2(r+1)/3]^{-3}< \CP 1/r$, Theorem~\ref{thm:Exist-mu} guarantees that a solution to
\eqref{eq:bdapprox} exists. Let $A_{r}=\lambda_{1}\mathbf{u}_{1}%
\otimes\mathbf{v}_{1}\otimes\mathbf{w}_{1}+\dots+\lambda_{r}\mathbf{u}%
_{r}\otimes\mathbf{v}_{r}\otimes\mathbf{w}_{r}$ be a solution and let
$U=\{\mathbf{u}_{1},\dots,\mathbf{u}_{r}\}$, $V=\{\mathbf{v}_{1}%
,\dots,\mathbf{v}_{r}\}$, $W=\{\mathbf{w}_{1},\dots,\mathbf{w}_{r}\}$. Since
$\mu(U)\leq\mu_{1}$, $\mu(V)\leq\mu_{2}$, $\mu(W)\leq\mu_{3}$,the harmonic
mean-geometric mean inequality yields%
\[
3\left[  \frac{1}{\mu(U)}+\frac{1}{\mu(V)}+\frac{1}{\mu(W)}\right]  ^{-1}%
\leq\sqrt[3]{\mu(U)\mu(V)\mu(W)}<\sqrt[3]{\mu_{1}\mu_{2}\mu_{3}} \PC \le\frac{3}{2(r+1)}, \CP
\]
the decomposition of $A_{r}$ is unique by Theorem~\ref{thm:Krus-mu}.
\end{proof}

In the context of our application in Section~\ref{subarrays-sec}, this corollary means that radiating sources can be uniquely localized if they are either (i)~sufficiently separated in space (angular separation viewed by a subarray, or by the array defined by translations between subarrays), or (ii)~in time (small sample cross correlations), noting that the scalar product between two time series is simply the sample cross correlation. Contrary to more classical approaches based on second or higher order moments, both conditions are not necessary here --- Corollary~\ref{cor:final} requires only that the \textit{product} between coherences be small.
In addition, there is no need for long data samples since the approach is deterministic; this is totally unusual in antenna array processing.
Cross correlations need to be sufficiently small among sources only for identifiability purposes but they are not explicitly computed in the identification process. Hence our model is robust with respect to short record durations. Observe also that the number of time samples can be as small as the number of sources. Lastly, an estimate of source time samples may be obtained from the tensor decomposition as a key byproduct of this deterministic approach.




\bibliographystyle{elsarticle-num}







\end{document}